\documentclass[11pt,reqno]{amsart}
\usepackage[utf8]{inputenc}
\usepackage[margin=1in]{geometry}

\usepackage{amsmath, amssymb, amsthm}
\usepackage{mathrsfs}
\usepackage{wasysym}
\usepackage{gensymb}
\usepackage{mathtools}
\usepackage{cancel}

\usepackage{xspace}
\usepackage{xifthen}
\usepackage{setspace}

\usepackage{hyperref}
\hypersetup{colorlinks=true, citecolor=blue, linkcolor=blue, urlcolor=blue, pdfstartview=FitH}
\usepackage{cleveref}
\usepackage{enumerate}
\usepackage{array}

\usepackage[dvipsnames]{xcolor}
\usepackage{ytableau}
\usepackage{soul}
\usepackage{mdframed}
\usepackage{tcolorbox}
\tcbuselibrary{theorems}
\mdfsetup{everyline=true}
\usepackage{framed}

\usepackage{tikz}
\usepackage{tikz-cd}
\usetikzlibrary{shapes.geometric}
\usetikzlibrary{decorations.pathmorphing}
\usepackage[all,2cell,cmtip]{xy}
\UseTwocells
\usepackage{tikz-qtree}
\usepackage{forest}

\usepackage{diagbox}

\usepackage{multicol}
\usepackage{bm}
\usepackage{bbm}

\usepackage{graphicx}
\usepackage{pdfpages}

\usepackage[colorinlistoftodos,color=red!60]{todonotes}

\setul{0.5ex}{0.3ex}
\setulcolor{Plum}

\usepackage{pifont}

\renewcommand*{\d}{\mathop{}\!d}

\newcommand{\abs}[1]{\left\lvert#1\right\rvert}
\newcommand{\de}{\delta}
\newcommand{\e}{\varepsilon}

\renewcommand*{\Re}{\operatorname{Re}}

\newcommand{\ZZ}{\mathbb{Z}}
\theoremstyle{plain} 
\newtheorem{theorem}{Theorem}[section]
\newtheorem*{theorem*}{Theorem}
\newtheorem{prop}[theorem]{Proposition}

\newtheorem{lemma}[theorem]{Lemma}
\newtheorem{corollary}[theorem]{Corollary}
\newtheorem{conjecture}[theorem]{Conjecture}

\newtheorem*{example}{Example}

\theoremstyle{remark}

\newtheorem*{remark}{Remark}

\usepackage[maxnames=10,minnames=10,style=numeric]{biblatex}
\renewbibmacro{in:}{}
\usepackage{enumitem}

\usepackage{caption}
\title{Unexpected Biases Between Congruence Classes for Parts in $k$-indivisible Partitions}

\author[F. Jackson]{Faye Jackson}
\address{University of Michigan, 530 Church St, Ann Arbor, MI 48109}
\email{alephnil@umich.edu}
\urladdr{\href{http://www-personal.umich.edu/~alephnil}{http://www-personal.umich.edu/~alephnil}}

\author[M. Otgonbayar]{Misheel Otgonbayar}
\address{Massachusetts Institute of Technology, 77 Massachusetts Avenue
Cambridge, MA 02139-4307}
\email{misheel@mit.edu}


\date{\today}

\makeatletter
\providecommand\@dotsep{5}
\def\listtodoname{TODOS: Not Done (Red), Cite (Orange), Blue (Notation)}
\def\listoftodos{\@starttoc{tdo}\listtodoname}
\makeatother

\newcommand{\Li}{\operatorname{Li}}

\newcommand{\lr}[1]{\left(#1\right)}

\newcommand{\linecomment}[1]{}
\renewcommand{\P}{\mathcal{P}}
\newcommand{\D}{\mathcal{D}}
\newcommand{\R}{\mathbb{R}}
\newcommand{\CC}{\mathbb{C}}

\newcommand{\Mod}[1]{\ (\mathrm{mod}\ #1)}
\DeclarePairedDelimiter{\floor}{\lfloor}{\rfloor}

\numberwithin{equation}{section}
\newcommand{\eeta}{\eta}
\newcommand{\rbar}{\overline{r}}

\crefname{equation}{equation}{equations}
\crefname{prop}{proposition}{propositions}
\Crefname{prop}{Proposition}{Propositions}

\begin{document}

\begin{abstract}
    For integers $k,t \geq 2$, and $1\leq r \leq t$ let $D_k^\times(r,t;n)$ be the number of parts among all $k$-indivisible partitions of $n$ (i.e., partitions where all parts are not divisible by $k$) of $n$ that are congruent to $r$ modulo $t$. Using Wright's circle method, we derive an asymptotic for $D_k^\times(r,t;n)$ as $n \to \infty$ when $k,t$ are coprime. The main term of this asymptotic does not depend on $r$, and so, in a weak asymptotic sense, the parts are equidistributed among congruence classes. However, inspection of the lower order terms indicates a bias towards different congruence classes modulo $t$. This induces an ordering on the congruence classes modulo $t$, which we call the $k$-indivisible ordering. \linecomment{These unexpected biases depend both on the size of $k$ and the congruence class of $k$ mod $t$.}We prove that for $k \geq \frac{6(t^2-1)}{\pi^2}$ the $k$-indivisible ordering matches the natural ordering. We also explore the properties of these orderings when $k < \frac{6(t^2-1)}{\pi^2}$. 
\end{abstract}

\maketitle


\section{Introduction}

A \textit{partition} $\lambda = (\lambda_1,\ldots,\lambda_\ell)$ of a positive integer $n$ is a nonincreasing sequence of positive integers which sum to $n$. We use the notation $\lambda \vdash n$ to say that $\lambda$ is a partition of $n$. The $\lambda_j$ are called the \textit{parts} of the partition $\lambda$. Counting functions related to partitions occur in nearly all fields of mathematics, and thus asymptotic expansions and exact formulae for these functions are treasured by the mathematical community. Hardy and Ramanujan began the analytic study of partition functions by studying the number of partitions of $n$, which we denote by $p(n)$. Their analytic methods yielded the asymptotic
\begin{equation}
    p(n) \sim \frac{1}{4n\sqrt{3}}e^{\pi\sqrt{\frac{2n}{3}}}.
    \label{eq:hi}
\end{equation}
Hardy and Ramanujan's proof is remarkable due to its invention of the circle method. This method allows one to extract the coefficients of a generating function by understanding its asymptotic behavior near singularities lying on the unit circle, and has seen wide application throughout analytic number theory since its inception.

Recently, this method has been used to understand how many parts among families of partitions of $n$ lie in some specified residue class mod $t$. This question was first explored by Beckwith and Mertens in \cite{beckwith2017,beckwith2015} for the family of all partitions, and was subsequently explored by Craig in \cite{craig} for the family of partitions into distinct parts. Formally, Craig defined\footnote{In \cite{craig}, the function $D(r,t;n)$ is denoted $D_{r,t}(n)$.}
\begin{align*}
	D(r,t;n) \coloneqq \sum_{\substack{\lambda \vdash n \\ \lambda \in \mathcal{D}}} \#\left\{\lambda_j \colon \lambda_j \equiv r \Mod t\right\},
\end{align*}
where $\mathcal{D}$ denotes the set of partitions with distinct parts, and he then proved the asymptotic 
\begin{equation}
    D(r,t;n) = \frac{3^{\frac{1}{4}}e^{\pi\sqrt{\frac{n}{3}}}}{2\pi tn^{\frac{1}{4}}}\left(\log(2) + \left(\frac{\sqrt{3}\log(2)}{8\pi} - \frac{t\pi}{4\sqrt{3}}\left(\frac{r}{t} - \frac{1}{2}\right)\right)n^{-\frac{1}{2}} + O(n^{-1})\right). \label{eq:craig-asym}
\end{equation}
More generally, we can consider partitions where no part is repeated $k$ or more times. These are referred to as the \textit{$k$-regular partitions}, and they are in bijection with the \textit{$k$-indivisible partitions}, which have no parts which are divisible by $k$. This is verified by the simple equality between their generating functions
\[
    \prod_{n \geq 1} \frac{\left(1 - q^{nk}\right)}{\left(1 - q^n\right)} = \prod_{n \geq 1} (1 + q^n + \cdots + q^{n(k-1)}).
\]
Often in the literature, these two types of partitions are both referred to as $k$-regular partitions. However, for the sake of clarity, we will always refer to the partitions with no part divisible by $k$ as the $k$-indivisible partitions. In a recent paper \cite{kregular}, the authors defined
\[
	D_k(r,t;n) \coloneqq \sum_{\substack{\lambda \vdash n \\ \lambda \in \mathcal{D}_k}} \#\left\{\lambda_j \colon \lambda_j \equiv r \Mod t\right\},
\]
where $\mathcal{D}_k$ denotes the set of $k$-regular partitions, and proved the asymptotic
\begin{equation}
	D_{k}(r,t;n) = \frac{3^{\frac{1}{4}}e^{\pi\sqrt{\frac{2Kn}{3}}}}{\pi t 2^{\frac{3}{4}}K^{\frac{1}{4}}n^{\frac{1}{4}}\sqrt{k}}\left(\log k + \left(\frac{3\sqrt{K}\log k}{8\sqrt{6}\pi} - \frac{t\pi(k-1)K^{\frac{1}{2}}}{2\sqrt{6}}\left(\frac{r}{t}- \frac{1}{2}\right)\right)n^{-\frac{1}{2}} + O(n^{-1})\right) \label{eq:regular-asym},
\end{equation}
where $K \coloneqq 1 - 1/k$. Notice that this is a strict generalization of Craig's work in \cite{craig}, as the $2$-regular partitions are exactly the partitions into distinct parts, and in fact the explicit error terms derived in \cite{kregular} improve upon those derived in \cite{craig}. In this paper, we derive an asymptotic formula for the number of parts congruent to $r$ mod $t$ among the $k$-indivisible partitions of $n$. We define
\begin{align}
	D_k^\times(r,t;n) \coloneqq \sum_{\substack{\lambda \vdash n \\ \lambda \in \mathcal{D}_k^\times}} \#\left\{\lambda_j \colon \lambda_j \equiv r \Mod t\right\} \label{eq:dk-times-defn},
\end{align}
where $\D_k^\times$ denotes the set of $k$-indivisible partitions. We prove the following asymptotic formula for $D_k^\times(r,t;n)$.
\begin{theorem}\label{thm:indiv-asym}
     Let $k,t \geq 2$ be coprime integers and let $1 \leq r \leq t$. If $K \coloneqq 1 - 1/k$ and $1 \leq \bar{r} \leq t$ is a representative of $k^{-1}r$ modulo $t$, then as $n \to \infty$,
     \begin{align*}
          D_k^\times(r,t;n) = \frac{3^{\frac{1}{4}}e^{\pi\sqrt{\frac{2Kn}{3}}}}{2^{\frac{3}{4}}K^{\frac{1}{4}}n^{\frac{1}{4}}\pi t\sqrt{k}}\left(\frac{K}{2}\log n + \left(- \psi\left(\frac{r}{t}\right) +k^{-1}\psi\left(\frac{\bar{r}}{t}\right)\right) + C_{k,t} + O\left(n^{-\frac{1}{2}}\log n\right)\right),
     \end{align*}
     where $\psi(x) \coloneqq \frac{\Gamma'(x)}{\Gamma(x)}$ is the digamma function and
     \[
        C_{k,t} \coloneqq \frac{K}{2}\log\left(\pi\sqrt{\frac{K}{6}}\right) - K \log t + \frac{\log k}{k}.
     \]
\end{theorem}

\begin{remark}
    Here we restrict to the case where $k,t$ are coprime for aesthetic reasons. Our method extends to the non-coprime case; however, this would require additional computation and casework while not revealing any deeper structure among the biases for $D_k^\times(r,t;n)$. For a more detailed analysis of this choice, see the discussion following \Cref{lemma:lrt-times-e}.
\end{remark}

\begin{example}
	\Cref{thm:indiv-asym} may also be used to obtain accurate numerical estimates for $D_k^\times(r,t;n)$. More precisely, let $\widehat{D}^\times_k(r,t;n)$ denote the asymptotic obtained in \Cref{thm:indiv-asym} by ignoring all terms which are $O\left(n^{-\frac{3}{4}}\log(n)e^{\pi\sqrt{\frac{2Kn}{3}}}\right)$, and define the quotient $Q_k^\times(r,t;n) \coloneqq \frac{D_k^\times(r,t;n)}{\widehat{D}^\times_k(r,t;n)}$. \Cref{fig:indiv-asym-numerics} shows the convergence of $Q_k^\times(r,t;n)$ to $1$ as $n \to \infty$.
\end{example}

\begin{figure}[ht]
    \centering
    \begin{tabular}{ | c |c | c | c | c | c | c |}
         \hline $n$ & 10 & 100 & 1000 & 10000 & 100000 & 1000000 \\
         \hline
         $Q_3^\times(1,4;n)$ & 0.95865 & 0.98376 &0.99054 & 0.99260 & 0.99355 & 0.99419 \\
         \hline
         $Q_3^\times(2,4;n)$ & 1.08452 & 0.99408 & 0.98952 & 0.98943 & 0.99044 & 0.99156\\
         \hline
         $Q_4^\times(1,5;n)$ & 0.92882 & 0.97154 & 0.98102 & 0.98437 & 0.98617 & 0.98746\\
         \hline
         $Q_4^\times(2,5;n)$ & 0.93232 & 0.96178 & 0.97154 & 0.97618 & 0.97947 & 0.98203 \\
         \hline
    \end{tabular}
    \caption{Numerics for \Cref{thm:indiv-asym}}
    \label{fig:indiv-asym-numerics}
\end{figure}

Because the main terms in \cref{eq:regular-asym} as well as in \Cref{thm:indiv-asym} do not depend on $r$, asymptotically the parts are equidistributed among congruence classes modulo $t$. Namely, if we let $P_k(n)$ (respectively, $P_k^\times(n)$) denote the total number of parts in $k$-regular partitions (respectively $k$-indivisible) partitions of $n$, then $\frac{D_k(r,t;n)}{P_k(n)}$ approaches $\frac{1}{t}$ as $n \to \infty$ and $\frac{D_k^\times(r,t;n)}{P_k^\times(n)}$ also approaches $\frac{1}{t}$ as $n \to \infty$ if $k,t$ are coprime. However, this weak asymptotic equidistribution does not imply that there are no biases between residue classes. In fact, analysis of the lower order terms uncovers the true nature of this bias. In particular, for \cref{eq:regular-asym}, $\frac{r}{t}$ is an increasing function in $1 \leq r \leq t$, which implies biases towards parts lying in lower congruence classes mod $t$. The asymptotic derived the authors thus implies that $D_k(r,t;n) \geq D_k(s,t;n)$ for $1 \leq r < s \leq t$ when $n$ is sufficiently large. This matches the results of Beckwith and Mertens in \cite{beckwith2017} concerning the family of all partitions as well as the results of Craig in \cite{craig} for the family of partitions into distinct parts. In contrast, the biases obtained for the $k$-indivisible partitions are not nearly as predictable, and depend greatly on both the size of $k$ as well as the residue class of $k$ mod $t$. For brevity, we define
\begin{equation}
    \psi_{k, t}(r) \coloneqq -\psi\left(\frac{r}{t}\right) + \frac1k \psi\left(\frac{\overline{r}}{t}\right) \label{eq:psi-kt-defn}.
\end{equation}
Using this notation, we have the following corollary exhibiting these biases.
\begin{corollary}\label{cor:indiv-diff}
     Let $k,t \geq 2$ be coprime integers and let $1 \leq s, r \leq t$. If $K \coloneqq 1 - 1/k$, then
     \begin{align*}
          D_k^\times(r,t;n) - D_k^\times(s,t;n) = \frac{3^{\frac{1}{4}}e^{\pi\sqrt{\frac{2Kn}{3}}}}{2^{\frac{3}{4}}K^{\frac{1}{4}}n^{\frac{1}{4}}\pi t\sqrt{k}}
          \left(\psi_{k, t}(r) - \psi_{k,t}(s) + O\left(n^{-\frac{1}{2}}\log n\right)\right).
     \end{align*}
     Furthermore, if we have
     \[ 
        \psi_{k,t}(r) > \psi_{k, t}(s),
     \]
     then for sufficiently large $n$ we have that $D_k^\times(r,t;n) > D_k^\times(s,t;n)$.
\end{corollary}
In light of these biases, we may define an ordering on the residue classes $\{1,\ldots,t\}$ mod $t$. For $1\leq r, s\leq t$, we write $r\prec_{k, t}s$ (resp. $r \succ_{k, t} s$) if $D_k^\times(r, t;n) < D_k^\times(s, t;n)$ (resp. $D_k^\times(r, t;n)>D_k^\times(s, t;n)$) for all sufficiently large $n$. Also, let $\mathscr{O}(t)$ be the number of distinct orderings of $\{1, \dots t\}$ induced by $\prec_{k, t}$ over all $k$ such that $k \geq 2$ and $\gcd(k,t)=1$.

\begin{remark}
    It is not immediately clear from \Cref{cor:indiv-diff} that $\prec_{k,t}$ is in fact a total ordering on $\{1,\ldots,t\}$, as we may have that $\psi_{k,t}(r) = \psi_{k,t}(s)$.  However, numerical evidence suggests that there are no such pairs $r,s$ for any coprime $k,t \geq 2$. Furthermore, deep work of Gun, Murty, and Rath in \cite{murty} related to the vanishing of $L$-functions has shown that there exists a $t_0$ such that the set $\{\psi(a/t) \mid \gcd(a,t) = 1\}$ is linearly independent over $\mathbb{Q}$ for any $t$ coprime to $t_0$. Thus, for $r,s$ coprime to $t$, and $t$ coprime to $t_0$, there are no equalities in the second order term.
\end{remark}

Based on the numerical evidence mentioned in the above remark, we make the following conjecture, which implies that for coprime $k,t \geq 2$ and $r \neq s$, either the congruence class $r$ or $s$ mod $t$ is strictly more common among $k$-indivisible partitions of $n$.

\begin{conjecture}\label{conj:second-order-bias}
    For coprime integers $k,t \geq 2$ and $1 \leq r < s \leq t$, we have that $\psi_{k,t}(r) \neq \psi_{k,t}(s)$.
\end{conjecture}

As we have established, due to the interaction between the multiplicative structures modulo $t$ and modulo $k$, the biases among congruence classes for $D_k^\times(r,t;n)$ are much more complex than the biases for $T(r,t;n)$ or $D_k(r,t;n)$. We now give an example to illustrate the complexity of the biases among $k$-indivisible partitions.

\begin{example}
    In \Cref{fig:indivis-biases} we give a table of all possible orderings of $\{1,\ldots,t\}$ by $\prec_{k,t}$ for $t = 7$. Notice that for $k = 2$, the congruence class $2 \Mod 7$ occurs in the fifth position. The bias against the residue class $2 \Mod 7$ may be explained by the fact that smallest allowed part which is $2 \Mod 7$ is $9$. Similarly, for $k = 6,13,10,20$ the transposition of $7,6 \Mod 7$ may be accounted for because the congruence class of $6 \Mod 7$ includes $6,13,20$ as its smallest members. However, for $k = 12$, there is a bias against the residue class $5 \Mod 7$ because its \textit{second} smallest member is excluded as a part in $12$-indivisible partitions. Thus the ordering is not induced by the natural ordering on integers, and even further, it is not even induced by the ordering on integers once we lift to the smallest allowed part in $k$-indivisible partitions. In total, we also see that the total number of orderings, $\mathscr{O}(7)$, is seven.
\end{example}

\begin{figure}[ht]
    \centering
    \begin{tabular}{ | c | c | c | c | c | c | c | c |}
        \hline 
        $k = 2$ & 1 & 3 & 5 & 7 & 2 & 4 & 6 \\
        \hline
        $k = 3$ & 1 & 2 & 4 & 5 & 7 & 3 & 6 \\ 
        \hline
        $k = 4$ & 1 & 2 & 3 & 5 & 6 & 7 & 4 \\
        \hline
        $k = 5$ & 1 & 2 & 3 & 4 & 6 & 7 & 5 \\
        \hline
        $k = 6,10,13,20$ & 1 & 2 & 3 & 4 & 5 & 7 & 6 \\
        \hline
        $k = 12$ & 1 & 2 & 3 & 4 & 6 & 5 & 7 \\
        \hline
        All other $k$  & 1 & 2 & 3 & 4 & 5 & 6 & 7 \\ \hline
    \end{tabular}
    \caption{Biases among congruence classes mod $t$ for $k$-indivisible partitions for $t = 7$, from most common to least common.}
    \label{fig:indivis-biases}
\end{figure}
Despite the intricacy present in \Cref{fig:indivis-biases}, closer inspection of the second order term $\psi_{k,t}$ reveals interesting patterns among the orderings.
We then have the following theorem.
\newline
\begin{theorem}\label{thm:k-indiv-biases}
    Let $k,t \geq 2$ be coprime integers. Then the orderings $\prec_{k,t}$ satisfy the following: 
    \begin{enumerate}[label={(\arabic*)},leftmargin=*]
        \item\label{item:r+k} If $1 \leq r \leq t - k$, then $r \succ_{k,t} r+k$.
        \item\label{item:fixing} If $1 \leq r \leq y \leq t$ and $r < s \leq t$, then for $k \geq y(y+1)$, $r \succ_{k,t} s$. Notably, when $y=1$, $1 \succ_{k, t} s$ holds for any $k, s, t \geq 2$.
        \item\label{item:t2} If $k \geq \frac{6(t^2-1)}{\pi^2}$, then for $1 \leq r < s \leq t$, we have $r \succ_{k,t} s$.
        \item\label{item:high-unnatural} If $k = mt - 1$ for $m \geq 1$, then for $k \leq \left(\frac{\pi^2}{6} + \frac{5}{2t}\right)^{-1}(t^2 - 1)$ we have that $t \succ_{k,t} t - 1$.
        \item\label{item:supralinear} If $t > 2$ and $\varphi(\cdot)$ is the Euler's totient function, then $\mathscr{O}(t) \geq \frac{\varphi(t)}{2}$.
    \end{enumerate}
\end{theorem}

\begin{remark}
    Observe that \vspace{0.1in}
    
    \noindent (i) Statement \ref{item:t2} indicates that when $k > \frac{6(t^2-1)}{\pi^2}$ the ordering reverts to the natural ordering observed by Beckwith and Mertens, Craig, and the authors in \cite{beckwith2017,craig,kregular}. Furthermore, \ref{item:high-unnatural} indicates that this bound for the maximum ordering which is not natural is asymptotically tight. This bound also implies that $\mathscr{O}(t) \leq \frac{6(t^2-1)}{\pi^2}$ and that, for any fixed $t$, a simple computer search will yield all the possible orderings arising from these biases. \vspace{0.1in}
    
    \noindent (ii) Using the techniques of Craig and the authors in \cite{craig,kregular} respectively, we could, in principle, make the error terms in \Cref{thm:indiv-asym} explicit, and thus for fixed $t$ find the precise $n$ where the ordering from \Cref{cor:indiv-diff} among congruence classes takes over. For brevity, we do not include this calculation.
\end{remark}

Numerics also suggest that the quotient $\frac{\mathscr{O}(t)}{\varphi(t)}$ grows sublinearly in $t$ and superlogarithmically in $t$ (see \Cref{fig:o(t)}). Thus we make the following conjecture.
\begin{conjecture}
\label{conj:superlinear}
    As $t \to \infty$, we have $\frac{\mathscr{O}(t)}{\varphi(t)} = o(t)$ and $\log t = o\left(\frac{\mathscr{O}(t)}{\varphi(t)}\right)$.
\end{conjecture}

\begin{figure}[ht]
    \centering
    \includegraphics[scale=0.5]{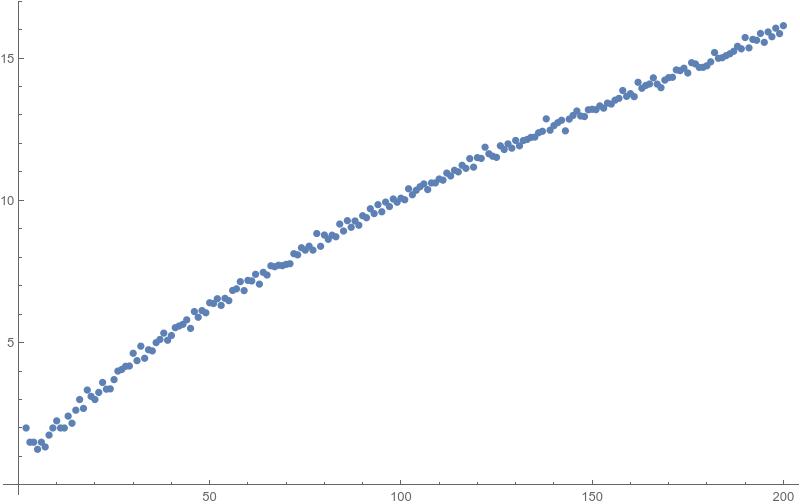}
    \caption{Graph of $\frac{\mathscr{O}(t)}{\varphi(t)}$ as $t$ ranges from $2$ to $200$}
    \label{fig:o(t)}
\end{figure}

\begin{remark}
    Although $\overline{r}$ is hard to predict for general $k$, it is fixed for $k$ in the same congruence classes modulo $t$. See the end of \Cref{subs:linear} for a discussion of this idea in relation to $\log t = o\left(\frac{\mathscr{O}(t)}{\varphi(t)}\right)$.
\end{remark}

We now describe the methods used to prove \Cref{thm:indiv-asym}. Following the methods established in \cite{beckwith2017,craig,kregular}, we make use of a distinct variation on Hardy and Ramanujan's circle method originally due to Wright (see for example \cite{bringmannCircle,ngo,wright}). Notably, the relevant generating function for $D_k^\times(r,t;n)$ is not modular, and so our application of the circle method differs significantly from traditional applications. However, as in \cite{beckwith2017,craig,kregular}, the generating function may be broken up into two components, the first of which is modular, and so may be estimated using traditional methods near one. The second component, which we refer to as the summatory component, is expressible as a sum of polylogarithms, and so Euler-Maclaurin summation may be used to compute asymptotic expansions. These estimates are then combined to produce the asymptotic expansion for $D_k^\times(r,t;n)$ via Wright's circle method. This portion of the paper broadly follows the techniques established by Beckwith and Mertens in \cite{beckwith2017} and used by Craig as well as the authors in \cite{craig,kregular}. As in \cite{beckwith2017}, the asymptotic for $D_k^\times(r,t;n)$ contains the digamma function $\psi(x) = \frac{\Gamma'(x)}{\Gamma(x)}$. For a discussion of why this function appears for $D_k^\times(r,t;n)$ but not for $D_k(r,t;n)$, see the remark following \Cref{lemma:etimes-series}.


The paper is organized as follows. In \Cref{sec:prelim}, we recall known results which we will use throughout the paper. These include the properties of the digamma function, two distinct variants on Wright's circle method (see \cite{bringmannCircle,ngo}), and an asymptotic version of Euler-Maclaurin summation due to Bringmann, Craig, Ono, and Males (see \cite{bringmannCircle}). In \Cref{sec:estimates}, we derive a convenient form of the generating function for $D_k^\times(r,t;n)$ and use Euler-Maclaurin summation as well as modularity to obtain the necessary bounds on the major and minor arcs for use in the circle method. \Cref{sec:indiv-asym} then applies the circle method with these asymptotics to prove \Cref{thm:indiv-asym}. Lastly, \Cref{sec:biases} uses properties of the digamma function to prove \Cref{thm:k-indiv-biases}.

\section*{Acknowledgements}

The authors were participants in the 2022 UVA REU in Number Theory. They would like to thank Ken Ono, the director of the UVA REU in Number Theory, as well as their graduate student mentor William Craig. They would also like to thank their colleagues at the UVA REU for their encouragement and support. They are grateful for the support of grants from the National Science Foundation (DMS-2002265, DMS-2055118, DMS-2147273), the National Security Agency (H98230-22-1-0020), and the Templeton World Charity Foundation.



\section{Preliminaries}\label{sec:prelim}

\subsection{Bernoulli Polynomials and Special functions}\label{subs:special function}
We recall the definition of the \textit{Bernoulli polynomials}. In \cite[(24.2.3)]{nist}, $B_n(x)$ is given by the Taylor expansion
\begin{align}
\sum_{n=0}^\infty B_n(x)\frac{t^n}{n!} \coloneqq \frac{te^{xt}}{e^t-1}.\label{eq:bernoulli}
\end{align}
The \textit{Bernoulli numbers} are the values of these polynomials at 0; i.e., $B_n\coloneqq B_n(0)$. We note that $B_{2n+1} = 0$ for $n>0$.

We also recall the definition of \textit{polylogarithms}: for $s \neq 1$ and $\abs q < 1$, define
\[
\Li_s(q) \coloneqq \sum_{n=1}^\infty \frac{q^n}{n^s},
\]
and for other $q$ define $\Li_s(q)$ by analytic continuation. Thus, its derivative satisfies the following property; if we set $q = e^{-z}$, then we have,
\begin{align}
    \frac{\partial \Li_s(q)}{\partial z} = -\Li_{s-1}(q).\label{eq:polylog-der}
\end{align}
We also need the expansion of $\Li_s$ at one. For $q = e^{-z}$ and $\abs z < 2\pi$, we have
\begin{align}
    \Li_s(q) = \Gamma(1-s)z^{s-1} + \sum_{n = 0}^\infty \zeta(s - n) \frac{(-z)^n}{n!} \label{eq:polylog-series}.
\end{align}

Now recall the \textit{digamma} function $\psi$, defined as the logarithmic derivative of $\Gamma$,
\[
\psi(x) \coloneqq \frac{\Gamma'(x)}{\Gamma(x)},
\]
where $\Gamma$ is defined as
\[
\Gamma(x) \coloneqq \int_0^\infty e^{-u}u^{x-1} \d u.
\]
We note the integral representation of $\psi(x)$, given in \cite[(5.9.12)]{nist} as
\begin{equation}
    \psi(x) = \int_0^\infty \frac{e^{-u}}{u} - \frac{e^{-xu}}{1 - e^{-u}}\d u \label{eq:integral-psi}
\end{equation}
for $\Re(x) > 0$. We also require the special values given below from \cite[(5.4.12)]{nist}, where $\gamma$ is the Euler-Mascheroni constant:
\begin{align}
    \psi(1) = -\gamma && \psi'(1) = \frac{\pi^2}{6} \label{eq:psi-special-value}.
\end{align}
We also require the recurrence relation given in \cite[(5.5.2)]{nist} as
\begin{equation}
    \psi(x + 1) = \psi(x) + \frac{1}{x} \label{eq:psi-recurrence}
\end{equation}
and the inequalities (see \cite{alzer})
\begin{equation}
    \log x - \frac{1}{x} \leq \psi(x) \leq \log x - \frac{1}{2x}, \label{eq:log-digamma}
\end{equation}
which hold for $x > 0$. These last two properties of the digamma function play a key role in establishing \Cref{thm:k-indiv-biases}, and thus in understanding $\prec_{k,t}$. Finally, note that $\psi$ is increasing on the interval $(0,\infty)$.

\subsection{Classical Results on the Partition Generating Function}\label{subs:classic-partitions}

Here, we note the properties of the \textit{partition generating function} given by
\[
    \P(q) \coloneqq \sum_{n \geq 0} p(n)q^n = \prod_{n \geq 1} \frac{1}{1 - q^n},
\]
which absolutely converges for $\abs{q} < 1$. For notational convenience, recall that the \textit{$q$-Pochhammer symbol} defined by
\[
    (a;q)_\infty \coloneqq \prod_{n \geq 1} \left(1 - aq^{n-1}\right).
\]
In particular, we may write $\P(q) = (q,q)_\infty^{-1}$. We now recall a bound and an asymptotic for $\P(q)$ when $0 < q < 1$, which can be found in Shakarchi and Stein's book on complex analysis as an exercise (see \cite[Ch.~10.4.1]{stein}). Namely we have for $0 < q < 1$ that
\begin{align}
    \log \P(q) &\leq \frac{\pi^2q}{6(1-q)} \label{eq:logp-absolute}, \\
    \log \P(q) &\sim \frac{\pi^2}{6(1-q)} \label{eq:logp-asym}.
\end{align}
The next exercise in \cite{stein} then uses \eqref{eq:logp-absolute} and \eqref{eq:logp-asym} to show the following subexponential bound on $p(n)$:
\begin{align}
    p(n) \leq e^{(\pi^2/6 + 1)\sqrt{n}} \label{eq:subexp-p}.
\end{align}
The bound \eqref{eq:subexp-p} reflects the shape of the asymptotic given by Hardy and Ramanujan given in \eqref{eq:hi}. We also require Euler's pentagonal number theorem, which will later allow us to bound $\P(q)$ when it appears in a denominator. This classical theorem takes the form
\begin{align}
    \P(q)^{-1} = \prod_{n \geq 1} (1 - q^n) = 1 + \sum_{m \geq 1} (-1)^m\left(q^{m(3m+1)/2} + q^{m(3m-1)/2} \right) \label{eq:pentagonal}.
\end{align}
Lastly, we state for later reference the modular transformation law for $\P$, which can be found in \cite[Thm.~5.1]{apostol}, and which is crucial for estimating the modular component of the generating function for $D_k^\times(r,t;n)$:
\begin{align}
    \P(e^{-z}) = \sqrt{\frac z{2\pi}} \exp\left(\frac{\pi^2}{6z} - \frac z{24}\right) \P\left(e^{-\frac{4\pi^2}z}\right). \label{eq:modular-p}
\end{align}

\subsection{Euler-Maclaurin Summation}\label{subs:euler-maclaurin}

In this subsection, we recall the asymptotic version of the classical Euler-Maclaurin summation due to Zagier \cite{zagier}, and another version due to Brignmann, Craig, Ono, and Males in \cite{bringmannCircle} that is suitable for application to functions with poles. Euler-Maclaurin summation quantifies the difference between the integral $\int_a^b f(x) \d x$ and an approximation of this integral, namely, the finite sum $f(a + 1) + \cdots + f(b)$. In particular, we have 
\begin{align*}
    \sum_{m=1}^{b-a} f(a + m) - \int_a^b f(x) \d x = \sum_{m=1}^N \frac{B_m}{m!}\left(f^{(m-1)}(b) - f^{(m-1)}(a)\right) + (-1)^{N+1}\int_a^b f^{(N)}(x)\frac{\widehat{B}_N(x)}{N!}\d x,
\end{align*}
where $\widehat{B}_n(x) \coloneqq B_n(x - \floor{x})$, with $\floor{x}$ denoting the greatest integer less than or equal to $x$. Zagier proved in \cite{zagier} that, if there exists an asymptotic expansion for $f(z)$ near $z = 0$, the Euler-Maclaurin summation formula may be used to give an asymptotic for $\sum_{m = 1}^\infty f(mz)$ as $z \to 0$.
Here, we mean asymptotic expansion in the strong sense, where we write $f(z) \sim \sum_{n \geq 0} b_nz^n$ if
\begin{align*}
f(z) - \sum_{n=1}^N b_nz^n = O(z^{N+1}) \ \ \ \  \text{as } z \to 0,
\end{align*}
for all $N > 0$. Many authors have applied this asymptotic form of Euler-Maclaurin summation in recent years to understand the growth of generating functions without nice modular transformation laws (for examples, see \cite{beckwith2017,bringmannCircle,bringmannTauberian,bringmannRankUni,craig,craig2021seaweed,kregular}).

We use the following notations freely throughout the paper. We also set
\begin{align*}
    I_f \coloneqq \int_0^\infty f(x) \d x
\end{align*}
for any function $f$ for which this integral converges. In \cite{zagier}, Zagier requires that $f$ have \textit{rapid decay at infinity}, that is $f(x) = O(x^{-N})$ as $x \to \infty$ for any $N > 1$. The version derived in \cite{bringmannCircle} by Bringmann, Craig, Ono, and Males requires a less restrictive decay condition on $f(x)$ at infinity, which they refer to as \textit{sufficient decay}, which holds if $f(x) = O(x^{-N})$ as $x \to \infty$ for some $N > 1$. Furthermore, they allow $f$ to have poles at $0$. Before stating this version, we state the original version due to Zagier.
\begin{prop}[{\cite[Proposition~3]{zagier}}]\label{prop:zagier-asym}
    Let $f$ be a $C^\infty$ function on the positive real line which has asymptotic expansion $f(x) \sim \sum_{n=0}^\infty c_nx^n$ at the origin, and together with all its derivatives, is of rapid decay at infinity. Then the function $g(x) = \sum_{m=1}^\infty f(mx)$ has asymptotic expansion
    \[
        g(x) \sim \frac{I_f}{x} + \sum_{n=0}^\infty c_n\frac{B_{n+1}}{n+1}(-x)^n
    \]
    as $x \to 0^+$.
\end{prop}

Note that this proposition only applies to functions on the real line. In contrast, the version in \cite{bringmannCircle} may be used for functions on the complex plane. However, we must take $z \to 0$ within some specified region in order for the asymptotic to hold. To make this precise, for $\theta > 0$, write
\[D_\theta \coloneqq \{z \in \mathbb{C} : \abs{\operatorname{arg} z} < \frac{\pi}{2} - \theta\}.\] 
Notice that if we set $z = \eeta + iy$ for $\eeta > 0$, then $z \in D_\theta$ if and only if $0 < \abs{y} < \Delta \eeta$ for some constant $\Delta > 0$ depending on $\theta$.

We now provide the asymptotic version of Euler-Maclaurin summation due to Bringmann, Craig, Males, and Ono which may be used when $f$ has a pole at zero.

\begin{prop}[{\cite[Lemma~2.2]{bringmannCircle}}]\label{prop:euler-mac-asym-pole}
    Let $0 < a \leq 1$ and $A \in \R^+$, and let $f(z) \sim \sum_{n = n_0}^\infty c_nz^n$ where $n_0 \in \ZZ$ as $z \to 0$ in $D_\theta$. Furthermore, assume that $f$ and all of its derivatives are of sufficient decay in $D_\theta$. Then we have that
    \begin{align*}
        \sum_{n = 0}^\infty f((n+a)z) \sim \sum_{n = n_0}^{-2} c_n\zeta(-n,a)z^n + \frac{I^\ast_{f,A}}{z} - \frac{c_{-1}}{z}\left(\log(Az) + \psi(a) + \gamma\right) - \sum_{n = 0}^\infty c_n\frac{B_{n+1}(a)}{n+1}z^n
    \end{align*}
    as $z \to 0$ uniformly in $D_\theta$, where
    \begin{align*}
        I_{f,A}^\ast \coloneqq \int_0^\infty \left(f(u) - \sum_{n = n_0}^{-2} c_nu^n - \frac{c_{-1}e^{-Au}}{u}\right) \d u,
    \end{align*}
    $\psi$ is the digamma function, $\gamma$ is the Euler-Mascheroni constant, and $\zeta(s,a) \coloneqq \sum_{n \geq 0} \frac{1}{(n+a)^s}$ is the Hurwitz zeta function.
\end{prop}

\subsection{Variants on Wright's Circle Method}\label{subs:wright-circle}

In this subsection, we recall variants on Wright's circle method which will be instrumental to our proof of \Cref{thm:indiv-asym}.  To begin, we recall the intuition which informs the method. Wright's circle method (see \cite{wright}), although inspired by Hardy-Ramanujan's circle method, differs from it by positing only that the generating function $F(q)$ is largest near a finite number of singularities, and that the contribution of the other singularities can be treated as an error terms in the integral. Thus, asymptotics for the coefficients of the $q$-series for $F(q)$ may be extracted using Cauchy's integral formula.

More precisely, given a circle $\mathcal{C}$ centered at the origin with radius less than one, we use the term \textit{major arc} for the region of $\mathcal{C}$ where $F(q)$ is largest, which should include its principal singularities. In particular, consider a circle of radius $e^{-\eeta}$ and let $\mathcal{L} \coloneqq \{\eeta + iy \mid \abs{y} \leq \pi\}$ so that $\mathcal{C} \coloneqq \{e^{-z} \mid z \in \mathcal{L}\}$. As in \cite{beckwith2017,craig,kregular}, the relevant generating function to our application of the circle method has a dominating singularity at $q = 1$. Thus, the major arc $\mathcal{C}_1$ is given by $\mathcal{C}_1 = \{e^{-z} \mid z \in \mathcal{L} \cap D_\theta\}$ for some $\theta>0$. The \textit{minor arc} of $\mathcal{C}$ is then defined by $\mathcal{C}_2 \coloneqq \mathcal{C} \setminus \mathcal{C}_1$. As mentioned above, the integral along the arc $\mathcal{C}_1$ is the main term in the asymptotic, and the integral along $\mathcal{C}_2$ is an error term.

Here, we recall the version of Wright's circle method due to Bringmann, Craig, Males, and Ono (see \cite{bringmannCircle}) which we will use in the proof of \Cref{thm:indiv-asym}.
\begin{theorem}[{\cite[Proposition~4.4]{bringmannCircle}}]\label{thm:simple-circle}
    Suppose that $F(q)$ is analytic for $q = e^{-z}$ where $z = x + iy$ satisfies $x > 0$ and $\abs{y} < \pi$, and suppose that $F(q)$ has an expansion $F(q) = \sum_{n=0}^\infty c(n)q^n$ near $q = 1$. Let $N,\Delta > 0$ be fixed constants. Consider the following hypotheses:
    \begin{enumerate}[label={(\arabic*)},leftmargin=*]
        \item\label{item:major-arc} As $z \to 0$ in the bounded cone $\abs{y} \leq \Delta x$ (major arc), we have
        \[
        F(e^{-z}) = Cz^Be^{\frac{A}{z}}\left(\sum_{j = 0}^{N-1} \alpha_jz^j + O_\theta(\abs{z}^N)\right),
        \]
        where $\alpha_s \in \CC, A,C \in \R^+,$ and $B \in \R$.
        \item\label{item:minor-arc} As $z \to 0$ in the bounded cone $\Delta x \leq \abs{y} < \pi$ (minor arc), we have
        \[
            \abs{F(e^{-z})} \ll_{\theta} e^{\frac{1}{\Re(z)}(A - \rho)}
        \]
        for some $\rho \in \R^+$.
    \end{enumerate}
    If \ref{item:major-arc} and \ref{item:minor-arc} hold, then as $n \to \infty$ we have
    \[
    c(n) = Cn^{\frac{1}{4}(-2B - 3)}e^{2\sqrt{An}}\left(\sum_{r=0}^{N-1} p_rn^{-\frac{r}{2}} + O\left(n^{-\frac{N}{2}}\right)\right),
    \]
    where $p_r \coloneqq \sum\limits_{j = 0}^r \alpha_j c_{j,r-j}$ and $c_{j,r} \coloneqq \frac{\left(-\frac{1}{4\sqrt{A}}\right)\sqrt{A}^{j + B + \frac{1}{2}}}{2\sqrt{\pi}} \cdot \frac{\Gamma(j + B + \frac{3}{2} + r)}{r!\Gamma(j + B + \frac{3}{2} - r)}$.
\end{theorem}

\begin{remark}
    Although the constant $C$ in \Cref{thm:simple-circle} does not appear in \cite{bringmannCircle}, we see that it is equivalent to the result in \cite{bringmannCircle} by factoring out $C$ from each $\alpha_i$.
\end{remark}

We also require Ngo and Rhoades' implementation of the circle method to functions of \textit{logarithmic type} (defined in \cite{ngo} and restated in the proposition below). Before stating this variant, we define for brevity
\[
    (v,m) \coloneqq \frac{\Gamma\left(v + m + \frac{1}{2}\right)}{m!\Gamma\left(v - m + \frac{1}{2}\right)}.
\]
For functions $L^\ast,\xi^\ast$ which are holomorphic in the unit disk, we also define the \textit{$q$-series coefficients of $L^\ast(q)\xi^\ast(q)$} by
\[
    V(n) = \frac{1}{2\pi i} \int_{\mathcal{C}} \frac{L^\ast(q)\xi^\ast(q)}{q^{n+1}} \d q.
\]
We then have the following from \cite{ngo}.
\begin{theorem}[{\cite[Proposition~1.9]{ngo}}]\label{thm:ngo-circle}
    Let $L^\ast,\xi^\ast$ be holomorphic functions within the unit disk satisfying the following hypotheses for $q = e^{-z}$, $z = \eeta + iy$ with $\eeta > 0, 0 \leq \abs{y} \leq \pi$:
    \begin{enumerate}[label={(\arabic*)},leftmargin=*]
        \item\label{item:hypo-1} For every positive integer $N$, as $\abs{z} \to 0$ in $D_\theta$, we have that
            \begin{align*}
            L^\ast(q) = \frac{\log z}{z^B}\left(\sum_{m = 0}^{N - 1} \alpha_m z^m + O_\theta(z^N)\right),
            \end{align*}
        where $\alpha_m \in \CC$ and $B$ is a real constant (in this case we say $L(q)$ has logarithmic type near $1$).
        \item\label{item:hypo-2} As $z \to 0$ in $D_\theta$, we have that
            \[
                \xi^\ast(q) = z^{B-1} e^{c^2/t}\left(1 + O\left(e^{-\gamma/z}\right)\right),
            \]
            where $\gamma$ is a positive real number.
        \item\label{item:hypo-3} As $z \to 0$ within the region $D_\theta' \coloneqq \{z \in \CC \mid \Re z > 0, z \notin D_\theta\}$,
        \[
            \abs{L^\ast(q)} \ll_{\theta} \eeta^{-C},
        \]
        where $C = C(\theta)$ is a positive real constant.
        \item\label{item:hypo-4} As $z \to 0$ within the region $D_\theta'$
        \[
            \abs{\xi^\ast(q)} \ll_\theta \xi(\abs{q})e^{-\epsilon/\eeta},
        \]
        where $\epsilon = \epsilon(\theta)$ is a positive real constant.
    \end{enumerate}
    Then we have an asympototic formula for the $q$-coefficients of $L^\ast(q)\xi^\ast(q)$
    \[
        V(n) = e^{2c\sqrt{n}}n^{-\frac{1}{4}}\left(\sum_{m=0}^{N-1} n^{-\frac{m}{2}}\left(c_r' + c_r\log n\right) + O\left(n^{-\frac{N}{2}} \log n\right)\right)
    \]
    where, with $\de_{m \geq 1} = 1$ when $m \geq 1$ and zero for $m < 1$, we set
    \begin{align*}
        a_m^\ast(s) &\coloneqq -\sum_{j = 0}^{\min(s-1,m)} \frac{\Gamma(s+1)(-2)^{2j-m}}{\Gamma(s-j)(j+1)(2\pi)^{\frac{1}{2}}}(s-1-j,m-j), \\
        \ell_{s,m} &\coloneqq \left(-\frac{1}{4c}\right)^m \frac{c^{s-\frac{1}{2}}}{-4\pi^{\frac{1}{2}}}(s,m), \\
        \ell_{s,m}' &\coloneqq \left(-\frac{1}{4c}\right)^m \frac{c^{s-\frac{1}{2}}\log c}{2\pi^{\frac{1}{2}}}(s,m) + \de_{m \geq 1} \frac{c^{s-m - \frac{1}{2}}}{2^{m + \frac{1}{2}}}a_{m-1}^\ast(s)
    \end{align*}
    and
    \begin{align*}
        c_m &\coloneqq  \sum_{s = 0}^m \alpha_s\ell_{s,m-s} && c_m' \coloneqq \sum_{s = 0}^m \alpha_s\ell_{s,m-s}'.
    \end{align*}
\end{theorem}

\begin{remark}
    In \cite{ngo}, hypothesis \ref{item:hypo-2} requires that $\gamma > c^2$. However, it is sufficient to have $\gamma > 0$ for the estimate arising from \ref{item:hypo-2} to be an error term. Similarly, in \cite{ngo}, the statements of hypotheses \ref{item:hypo-3} and \ref{item:hypo-4} are written as
    \begin{align*}
        L^\ast(q) &\ll_\theta z^{-C} \\
        \abs{\xi^\ast(q)} &\ll_\theta \xi^\ast(\abs{q})e^{-\epsilon/z}.
    \end{align*}
    However, these are simple typographical errors, as can be discovered by closely reading the proof.
\end{remark}
%
%

\section{Estimates on the Major/Minor Arcs}\label{sec:estimates}

\subsection{Generating Functions}\label{subs:indiv-gen}

In this subsection, we break the generating function $\D^\times_k(r,t;q)$ of $D^\times_k(r,t;q)$ into its modular and summatory components, and then we write the latter in a form to which we can apply Euler-Maclaurin summation. Let
\begin{align*}
    \D^\times_k(r,t;q) \coloneqq \sum_{n \geq 0} D_k^\times(r,t;q)q^n
\end{align*}
be the generating function of $D_k^\times(r,t;n)$. We then have the following expression for $\D_k^\times(r,t;q)$.

\begin{lemma}\label{lemma:dtimes-gen}
    We have that
    \begin{align*}
        \D_k^\times(r,t;q) = \frac{(q^k;q^k)_\infty}{(q;q)_\infty}\sum_{\substack{k \,\nmid m\, \\ m \equiv r \Mod t}} \frac{q^m}{1 - q^m},
    \end{align*}
    and if $\gcd(k,t) = 1$ then
    \begin{align*}
        \D_k^\times(r,t;q) = \frac{(q^k;q^k)_\infty}{(q;q)_\infty}\left(\sum_{\substack{m \equiv r \Mod t}} \frac{q^m}{1 - q^m} - \sum_{m \equiv \bar{r}_{k,t} \Mod t} \frac{q^{mk}}{1 - q^{mk}}\right),
    \end{align*}
    where $1 \leq \bar{r}_{k,t} \leq t$ is the representative of $rk^{-1} \Mod t$.
\end{lemma}

\begin{proof}
    First note that when $\gcd(k,t) = 1$ the second equality follows from the first by writing
    \begin{align*}
        \sum_{\substack{k \, \nmid m \, \\ m \equiv r \Mod t}} \frac{q^m}{1 - q^m} = \sum_{m \equiv r \Mod t} \frac{q^m}{1 - q^m} - \sum_{\substack{k \mid m \\ m \equiv r \Mod t}} \frac{q^m}{1 - q^m}
    \end{align*}
    and changing variables. It is a classical fact that $\xi_k(q) \coloneqq (q^{k+1},q^{k+1})_\infty(q,q)^{-1}_\infty$ is the generating function for $k$-indivisible partitions. Now for $k \nmid m$ define $D_{k,m}^\times(q)$ as the generating function for $k$-indivisible partitions where each partition is weighed by how many parts of size $m$ it contains. We see that
    \begin{align*}
        D_{k,m}^\times(q) &= \xi_k(q) \cdot \frac{\sum_{j \geq 1} jq^{jm}}{\sum_{j \geq 0} q^{jm}} = \xi_k(q) \cdot \frac{q^m(1-q^m)}{(1-q^m)^2} = \xi_k(q) \cdot \frac{q^m}{1-q^m}.
    \end{align*}
    Recalling \cref{eq:dk-times-defn}, the result then follows by summing over $m \equiv r \Mod t$ for $k \nmid m$.
\end{proof}

We break this generating function up into the following components:
\begin{align*}
    \xi_k(q) \coloneqq \frac{(q^k,q^k)_\infty}{(q,q)_\infty}, && L_k^\times(r,t;q) \coloneqq \sum_{\substack{k \, \nmid m \, \\ m \equiv r \Mod t}} \frac{q^m}{1 - q^m}.
\end{align*}
We call $\xi_k(q)$ the \textit{modular component} and $L_k(r,t;q)$ the \textit{summatory} component.\footnote{In Beckwith and Mertens work (see \cite{beckwith2017}), the summatory component $L$ is multiplied by $(2\pi)^{-1/2}q^{1/24}$ to easily apply Ngo-Rhodes' variant of the circle method. The notation in this paper matches that of Craig in \cite{craig} and the authors in \cite{kregular}. }

\begin{remark}
    Note that $\xi_k(q)$ exactly matches that found in the authors' previous work \cite{kregular} because the number of $k$-regular partitions and the number of $k$-indivisible partitions are in fact the same. Thus, we may apply many of the same estimates derived in \cite{kregular} here. As the estimates derived in \cite{kregular} are explicit and those needed in this paper are only asymptotic, we provide alternate proofs of these estimates for ease of reading.
\end{remark}
We thus first concern ourselves with $L^\times_k(r,t;q)$. For brevity, we will suppress the dependence of $\bar{r}_{k,t}$ on $k,t$ where these are clear. Now define $E_\times(z) \coloneqq \frac{e^{-z}}{1 - e^{-z}} = \Li_0(q)$. We now see that when $\gcd(k,t) = 1$, $L_k^\times(r,t;q)$ may be expressed as two sums over integers of $E_\times$ evaluated at specific values.

\begin{lemma}\label{lemma:lrt-times-e}
    When $\gcd(k,t) = 1$, we have
    \begin{align*}
        L_k^\times(r,t;q) = \sum_{\ell \geq 0} E_\times((\ell t + r)z) - \sum_{\ell \geq 0} E_\times((\ell t + \bar{r})zk).
    \end{align*}
\end{lemma}

\begin{proof}
    Immediate from the definitions of $L_k^\times(r,t;q)$ and $E_\times(z)$.
\end{proof}

\Cref{lemma:lrt-times-e} will allow us to use Euler-Maclaurin summation to estimate $L_k^\times(r,t;q)$ on the major arc when $k,t$ are coprime using similar techniques to those found in \cite{beckwith2017,craig,kregular}. On the other hand, if $\gcd(k,t) = d > 1$, then we can reduce to when $k,t$ are coprime. In particular, the $\xi_k$ component does not change in this case, and the summatory component changes in a predictable way. This involves casework on $d$ and on $r$. In these cases, the summatory component is either zero, matches that of Beckwith and Mertens,\footnote{As before, Beckwith and Mertens include a factor of $(2\pi)^{-1/2}q^{1/24}$} or may be written as $L_{k/d}^\times(r/d,t/d,q^d)$. Thus, the methods established in this paper and in \cite{beckwith2017} would be sufficient to analyze the behavior of the biases when $\gcd(k,t) > 1$.


We will now provide a polylogarithm expansion for $E_\times(z)$, as well as a series expansion near zero which we may apply when performing Euler-Maclaurin summation.

\begin{lemma}\label{lemma:etimes-series}
    We have that
    \begin{align}
        E_\times^{(N)}(z) = (-1)^N\Li_{-N}(q) = \frac{N!(-1)^N}{z^{N+1}} + \sum_{m = 0}^\infty \frac{B_{N+m+1}}{(N+m+1) \cdot m!}z^m,
    \end{align}
    where the second equality only holds when $\abs{z} < 2\pi$.
\end{lemma}

\begin{proof}
    Immediate from \eqref{eq:polylog-series} and \eqref{eq:polylog-der} as well as the special values $\zeta(-N-m) = (-1)^{N+m}\frac{B_{N+m+1}}{N+m+1}$.
\end{proof}

\begin{remark}
    $E_\times$ and all its derivatives have rapid decay within any region $D_\theta$. However, $E_\times^{(N)}$ does have a pole of order $N + 1$ at zero, which through \Cref{prop:euler-mac-asym-pole} ultimately leads to the digamma function appearing in the asymptotic derived in \Cref{thm:indiv-asym}. In contrast, the pole cancels in the analogue of $E_\times$ within \cite{kregular}.
\end{remark}

Specializing \Cref{lemma:etimes-series} to $N = 0$, we have a series expansion for $E_\times$ near zero, by which we define constants $e_{m}$ by
\begin{align*}
    E_\times(z) = \frac{e_{-1}}{z} + \sum_{m \geq 0} \frac{e_{m}}{m!} z^m \coloneqq \frac{1}{z} + \sum_{m \geq 0} \frac{B_{m+1}}{(m+1) \cdot m!} z^m.
\end{align*}

We now state for later use the transformation law for $\xi_k(q)$, which may also be found in \cite[Lemma~3.4]{kregular}.

\begin{lemma}\label{lemma:xi-transform}
    For $q = e^{-z}$ and $\e \coloneqq \exp\left(-\frac{4\pi^2}{kz}\right)$, we have that
    \[
        \xi_k(q) = \frac{1}{\sqrt{k}}\exp\lr{\frac{\pi^2}{6z}\lr{1 -\frac{1}{k}} + \frac{z}{24}(k-1)}\frac{\P(\e^k)}{\P(\e)}.
    \]
\end{lemma}

\begin{proof}
    This follows immediately from the modular transformation law \eqref{eq:modular-p} and the expression $\xi_k(q) = \frac{\P(q)}{\P(q^k)}$.
\end{proof}

\subsection{Asymptotics on the Major Arc}\label{subs:major-arc}


In this subsection, we provide the necessary asymptotics for $\xi_k(q)$ and $L_k^\times(r,t;q)$ on the major arc which we will need to prove \Cref{thm:indiv-asym}. The arguments for $L_k^\times(r,t;q)$ follow similarly to those for $L_k(r,t;q)$ in \cite{kregular}, but they require the use of \Cref{prop:euler-mac-asym-pole} because $E_\times$ has a pole at zero. This leads to the appearence of the digamma function $\psi$ in the asymptotic derived in \Cref{thm:indiv-asym}.

We now obtain the asymptotic for $L^\times_k(r,t;q)$ on the major arc. Let $\Delta > 0$ be fixed, and set $\de \coloneqq \sqrt{1 + \Delta^2}$. Furthermore from here on we refer to $D_\theta \coloneqq \{\eeta + iy \mid \eeta > 0, \abs{y} \leq \Delta \eeta\}$ as the major arc.

\begin{lemma}\label{lemma:e-times-major-arc}
    Let $0 < a \leq 1$, then
    \begin{align*}
        \sum_{m \geq 0} E_\times((m+a)z) = -\frac{\log z}{z} - \frac{\psi(a)}{z} + O(1)
    \end{align*}
    as $z \to 0$ uniformly in $D_\theta$.
\end{lemma}

\begin{proof}
    We apply \Cref{prop:euler-mac-asym-pole} to $E_\times$ with $A = 1$ using the series expansion derived in \Cref{lemma:etimes-series}. Because $E_\times$ has principal part $\frac{1}{z}$, this yields
    \begin{align*}
        \sum_{m \geq 0} E_\times((m+a)z) = \frac{I_{E_\times,1}^\ast}{z} - \frac{1}{z}\left(\log(z) + \psi(a) + \gamma\right) + O(1).
    \end{align*}
    We now must compute $I^\ast_{E_\times,1}$, this is precisely
    \begin{align*}
        I^\ast_{E_\times,1} = \int_0^\infty \frac{e^{-u}}{1 - e^{-u}} - \frac{e^{-u}}{u} \d u.
    \end{align*}
    By the integral representation for $\psi$ provided in \cref{eq:integral-psi} we have that $I^\ast_{E_\times,1} = -\psi(1)$. We also see that $\psi$ takes on the special value $\psi(1) = \gamma$ from \cref{eq:psi-special-value}. This concludes the proof.
\end{proof}

We may now apply \Cref{lemma:e-times-major-arc} to each portion of $L^\times_k(r,t;q)$ separately.

\begin{corollary}\label{cor:lrt-times-major-arc}
    Let $k,t \geq 2$ be coprime integers and let $1 \leq r \leq t$, then 
    \begin{align*}
        L^\times_k(r,t;q) = -\frac{K\log z}{tz} + \frac{1}{tz}\left(k^{-1}\psi\left(\frac{\bar{r}}{t}\right) - \psi\left(\frac{r}{t}\right) - K \log t + \frac{\log k}{k}\right) + O(1),
    \end{align*}
    for $K \coloneqq 1 - 1/k$, as $z \to 0$ uniformly in $D_\theta$.
\end{corollary}

\begin{proof}
    Write $L^\times_k(r,t;q)$ as in \Cref{lemma:lrt-times-e} as
    \begin{align*}
        L^\times(r,t;q) = \sum_{m \geq 0} E_\times\left(\left(m + \frac{r}{t}\right)tz\right) - \sum_{m \geq 0} E_\times\left(\left(m + \frac{\bar{r}}{t}\right)tkz\right)
    \end{align*}
    and then apply \Cref{lemma:e-times-major-arc} with $a = \frac{r}{t}, \frac{\bar{r}}{t}$ and the changes of variables $z = tz,tkz$ respectively.
\end{proof}

For ease of reading, we now prove an asympototic version of Lemma 3.8 from \cite{kregular}. To do so, we first require a basic lemma concerning exponential sums and geometric series, which appears as Lemma 3.7 in \cite{kregular}. 

\begin{lemma}\label{lemma:mvt-geom}
    Suppose we have a series of the form $\sum_{m \geq b} e^{f(x)}$ such that $f'(x)$ is decreasing and $f'(b) < 0$ then
    \[
        \sum_{m \geq b} e^{f(x)} \leq \frac{e^{f(b)}}{1 - e^{f'(b)}}.
    \]
\end{lemma}

\begin{proof}
    By the mean value theorem, for every $m \geq b$, $f(m+1)-f(m) = f'(c) \leq f'(b)$ for $c$ between $m$ and $m + 1$. Therefore, using the geometric series formula,
    \begin{align*}
        \sum_{m \geq b} e^{f(x)} \leq \sum_{m \geq b} e^{f(b)}e^{(m-b)f'(b)} = \frac{e^{f(b)}}{1-e^{f'(b)}}. \tag*{\qedhere}
    \end{align*}
\end{proof}
We may now prove the required asymptotic
\begin{lemma}\label{lemma:major-arc-xi}
    Let $z = \eeta + iy$ be any complex number $0 \leq \abs{y} \leq \Delta \eeta$, then as $\eeta \to 0$ we have
    \begin{align*}
        \xi_k(q) = \Phi_k(z)\left(1 + O(e^{-\rho/z})\right),
    \end{align*}
    where $\rho > 0$ and
    \begin{align*}
        \Phi_k(z) \coloneqq \frac{1}{\sqrt{k}}\exp\left(\frac{\pi^2}{6z}\left(1 - \frac{1}{k}\right) + \frac{z}{24}(k-1)\right).
    \end{align*}
\end{lemma}

\begin{proof}
    Recalling that $\e = e^{-4\pi^2/kz}$, \Cref{lemma:xi-transform} and Euler's pentagonal number theorem \eqref{eq:pentagonal} gives
    \begin{align*}
        \xi_k(q) = \Phi_k(z) \cdot \frac{\P(\e^k)}{\P(\e)} = \Phi_k(z) \cdot \left(1 + \sum_{m \geq 1} (-1)^m\left(\e^{\frac{m(3m+1)}{2}} + \e^{\frac{m(3m-1)}{2}}\right)\right)\left(1 + \sum_{m \geq 1} p(m)\e^{km}\right).
    \end{align*}
    Applying \cref{eq:subexp-p}, we may write
    \[
        \abs{\sum_{m \geq 1} p(m)\e^{km}} \leq \sum_{m \geq 1} e^{2.7\sqrt{m}}\abs{\e}^{km} = \sum_{m \geq 1} e^{2.7\sqrt{m} - 4\pi^2m\Re(1/z)}.
    \]
    Using that $\Re(1/z) = \frac{\eeta}{\abs{z}^2} \geq \frac{1}{\de \eeta}$ and differentiating with respect to $m$ yields
    \begin{align*}
        \frac{\d}{\d m} \left(2.7\sqrt{m} - 4\pi^2m\Re(1/z)\right) = \frac{1.35}{\sqrt{m}} - 4\pi^2\Re\left(\frac{1}{z}\right) \leq 1.35 - \frac{4\pi^2}{\de\eeta}.
    \end{align*}
    For small enough $\eeta$, this quantity is negative, and so \Cref{lemma:mvt-geom} applies to yield $\P(\e^k) - 1 = O(\e^k)$. Therefore, we have that
    \begin{align*}
        \xi_k(q) &= \Phi_k(q) \cdot (1 + O(\e))(q + O(\e^k)) = \Phi_k(q)(1 + O(e^{-\rho/z})),
    \end{align*}
    where $\rho = -4\pi^2/k > 0$.
\end{proof}

\subsection{Asymptotics on the Minor Arc}\label{subs:minor-arc}

In this subsection, we bound $L^\times_k(r,t;q)$ and $\xi_k(q)$ on the minor arc. We begin by bounding $L^\times_k(r,t;q)$.

\begin{lemma}\label{lemma:lrt-times-minor-arc}
    Let $k,t \geq 2$ and $1 \leq r \leq t$. Assuming that $z = \eeta + iy$ satisfies $\eeta > 0$, then we have
    \begin{align*}
        \abs{L_k^\times(r,t;q)} \leq \frac{e^{\eeta}}{\eeta^2}.
    \end{align*}
\end{lemma}

\begin{proof}
    By the triangle inequality and arguing similarly to Lemma 3.10 in \cite{kregular}, we have that
    \begin{align*}
        \abs{L^\times(r,t;q)} &\leq \sum_{m \geq 1} \frac{\abs{q}^m}{1 - \abs{q}^m} = \sum_{m \geq 1} \sigma_0(m)\abs{q}^m,
    \end{align*}
    where $\sigma_0$ denotes the number of positive integer divisors of $m$. Because $\sigma_0(m) \leq m$, we have
    \begin{align*}
        \abs{L^\times(r,t;q)} \leq \sum_{m \geq 1} mq^m = \frac{\abs{q}}{(1 - \abs{q})^2}.
    \end{align*}
    Applying the tangent bound for $e^{-z}$, we see that
    \begin{align*}
        \frac{\abs{q}}{(1 - \abs{q})^2} = \frac{e^{\eeta}}{(e^{\eeta} - 1)^2} \leq \frac{e^{\eeta}}{\eeta^2}. \tag*{\qedhere}
    \end{align*}
\end{proof}

In order to apply Ngo and Rhodes version of Wright's circle method (see \Cref{thm:ngo-circle}), we require the following estimate for $\abs{\xi_k(q)}$ in terms of $\xi_k(\abs{q})$ on the minor arc.
\begin{lemma}\label{lemma:minor-arc-xi-ngo}
    As $z = \eeta + iy \to 0$ within the region $\Delta \eeta \leq \abs{y} \leq \pi$ we have that
    \[
        \abs{\xi_k(q)} \ll_{k,\Delta} \xi_k(\abs{q})e^{\frac{-\epsilon}{\eeta}},
    \]
    for some $\epsilon > 0$, so long as $\Delta$ is sufficiently large.
\end{lemma}

\begin{proof}
    First we apply \Cref{lemma:xi-transform} to see that
    \begin{align*}
        \xi_k(q) &= \sqrt{\frac{1}{k}}\exp\left(\frac{\pi^2}{6z}\left(1 - \frac{1}{k}\right) + \frac{z}{24}(k-1)\right)\frac{\P(e^{-4\pi^2/z})}{\P(e^{-4\pi^2/kz})}.
    \end{align*}
    Taking absolute values then yields, for $K \coloneqq 1 - 1/k$, that
    \begin{align*}
        \abs{\xi_k(q)} &= \sqrt{\frac{1}{k}}\exp\left(\frac{\pi^2K\Re(1/z)}{6}+ \frac{\eeta}{24}(k-1)\right)\frac{\abs{\P(e^{-4\pi^2/z})}}{\abs{\P(e^{-4\pi^2/kz})}} \\
        &\leq \sqrt{\frac{1}{k}}\exp\left(\frac{\pi^2K\Re(1/z)}{6} + \frac{\eeta}{24}(k-1)\right)\frac{\P(e^{-4\pi^2\Re(1/z)})}{\abs{\P(e^{-4\pi^2/kz})}}.
    \end{align*}
    Now note that, as we are on the minor arc, $\Re\left(\frac{1}{z}\right) = \frac{\eeta}{\abs{z}^2} \leq \frac{1}{\de^2\eeta}$. Thus, we have that
    \begin{align*}
        \exp\left(\frac{\pi^2K\Re(1/z)}{6} + \frac{\eeta}{24}(k-1)\right) \leq \Phi_k(\eeta)\exp\left(-\frac{\pi^2K}{6\eeta}\left(1 - \frac{1}{\de^2}\right)\right).
    \end{align*}
    In light of this inequality, set
    \begin{align*}
        \epsilon \coloneqq \frac{\pi^2K}{6}\left(1 - \frac{1}{\de^2}\right) > 0.
    \end{align*}
    We now apply the transformation law given in \Cref{lemma:xi-transform} once more along with the fact that $\abs{q} = e^{-\eeta}$ to yield
    \begin{align*}
        \abs{\xi_k(q)} \leq \Phi_k(\eeta) e^{-\epsilon/\eeta} \frac{\P(e^{-4\pi^2\Re(1/z)})}{\abs{\P(e^{-4\pi^2/kz})}} = \xi_k(\abs{q}) e^{-\epsilon/\eeta} \frac{\P(e^{-4\pi^2/\eeta})}{\P(e^{-4\pi^2/\eeta k})} \cdot \frac{\P(e^{-4\pi^2\Re(1/z)})}{\abs{\P(e^{-4\pi^2/kz})}}.
    \end{align*}
    We may then choose $\de > 1$ sufficiently large such that $\Re\left(\frac{1}{z}\right) \leq \frac{1}{\eeta k}$. For such $\de$, we have
    \begin{align*}
        \abs{\xi_k(q)} \leq \xi_k(\abs{q})e^{-\epsilon/\eeta} \frac{\P(e^{-4\pi^2/\eeta})}{\abs{\P(e^{-4\pi^2/kz})}}.
    \end{align*}
    Applying \Cref{eq:logp-absolute} furnishes the inequality
    \begin{align*}
        \P\left(e^{-4\pi^2/\eeta}\right) \leq \exp\left(\frac{\pi^2e^{-4\pi^2/\eeta}}{6(1 - e^{-4\pi^2/\eeta})}\right) \leq \exp\left(\frac{\eeta}{24}\right) = O(1),
    \end{align*}
    using the tangent bound for $e^x$. To bound the denominator, we may again use Euler's pentagonal number theorem (as in \Cref{lemma:major-arc-xi}) to see that
    \begin{align*}
         \abs{\P(e^{-4\pi^2/kz})}^{-1} = 1 + O\left(e^{-4\pi^2/kz}\right).
    \end{align*}
    We may then use the inequality $\Re\left(\frac{1}{z}\right) = \frac{\eeta}{\abs{z}^2} \geq \frac{\eeta}{\pi^2 + \eeta^2} > \frac{0.99\eeta}{\pi^2}$ for small $\eeta$ to yield that
    \[
         \abs{\P(e^{-4\pi^2/kz})}^{-1} = 1 + O\left(e^{-3.96\eeta/k}\right) = O(1).
    \]
    Combining all of the above yields the desired estimate
    \begin{align*}
        \abs{\xi_k(q)} \leq \xi_k(\abs{q})\cdot O\left(e^{- \frac{\epsilon}{\eeta}}\right).
    \end{align*}
\end{proof}

\section{Proof of \texorpdfstring{\Cref{thm:indiv-asym}}{Theorem 1.5}}\label{sec:indiv-asym}

Here we apply \Cref{thm:simple-circle,thm:ngo-circle} to $\D_k^\times(r,t;q)$ to obtain \Cref{thm:indiv-asym}.  To apply \Cref{thm:simple-circle,thm:ngo-circle}, we break up our series as
\begin{align}
    \xi_k(q)\left(L_k^\times(r,t;q) + \frac{K\log z}{tz}\right) &\eqqcolon \sum_{n \geq 0} a_k(r,t;n)q^n \eqqcolon \mathcal{A}_k(r,t;q) \label{eq:a-poly-type} \\
    \xi_k(q)\frac{K\log z}{tz} &\eqqcolon \sum_{n \geq 0} c_k(t;n)q^n \label{eq:c-log-type}.
\end{align}
so that $D_k^\times(r,t;n) = a_k(r,t;n) - c_k(t;n)$ by \Cref{lemma:dtimes-gen}. Notice that \Cref{thm:simple-circle} directly applies to \eqref{eq:a-poly-type}, whereas \Cref{thm:ngo-circle} applies to \eqref{eq:c-log-type}.

\begin{lemma}\label{lemma:akrtn}
    Let $k,t \geq 2$ and let $1 \leq r \leq t$. We have
    \begin{align*}
        a_k(r,t;n) = \frac{3^{1/4}e^{\pi\sqrt{\frac{2Kn}{3}}}}{2^{3/4}K^{1/4}n^{1/4}\pi t\sqrt{k}}\left(k^{-1}\psi\left(\frac{\bar{r}}{t}\right) - \psi\left(\frac{r}{t}\right) + \frac{\log k}{k} - K\log t + O(n^{-1/2})\right)
    \end{align*}
    as $n \to \infty$, where $K \coloneqq 1 - 1/k$.
\end{lemma}

\begin{proof}

To apply \Cref{thm:simple-circle} to \eqref{eq:a-poly-type}, we must show that
\begin{align*}
    \mathcal{A}_k(r,t;q) = \frac{1}{tz\sqrt{k}}e^{\frac{\pi^2K}{6z}}\left(k^{-1}\psi\left(\frac{\bar{r}}{t}\right) - \psi\left(\frac{r}{t}\right) - K\log t + \frac{\log k}{k} + O(z)\right)
\end{align*}
on the major arc (see condition \ref{item:major-arc}). This follows directly from \Cref{cor:lrt-times-major-arc} and \Cref{lemma:major-arc-xi}. We must also show on the minor arc that
\begin{align*}
    \abs{\mathcal{A}_k(r,t;q)} \ll_{\Delta} e^{\frac{1}{\Re(z)}\left(\frac{\pi^2K}{6} - \rho\right)}
\end{align*}
for some $\rho > 0$ (see condition \ref{item:minor-arc}). This follows directly from \Cref{lemma:lrt-times-minor-arc,lemma:minor-arc-xi-ngo} for sufficiently large $\de > 1$. \Cref{thm:simple-circle} then implies the result.
\end{proof}

We will now apply \Cref{thm:ngo-circle} in order to estimate $c_k(t;n)$.

\begin{lemma}\label{lemma:cktn}
    Let $k,t \geq 2$ and $N > 0$. We have
    \begin{align*}
        c_k(t;n) = \frac{3^{1/4}e^{\pi\sqrt{\frac{2K}{3}n}}}{2^{3/4}K^{1/4}n^{1/4}\pi t\sqrt{k}}\lr{-\frac{K}{2}\log\lr{\pi\sqrt{\frac{K}{6}}} - \frac{K}{2}\log\lr{n} + O(n^{-1/2}\log n)}
    \end{align*}
    as $n \to \infty$, where $K \coloneqq 1 - 1/k$.
\end{lemma}

\begin{proof}
    As before, let $\eeta > 0$, $z = \eeta + iy$, $q = e^{-z}$, and define the major arc by $0 \leq \abs{y} \leq \Delta\eeta$ and the minor arc by $\Delta\eeta \leq \abs{y} \leq \pi$. For convenience, we rescale $\xi_k(q)\frac{K\log z}{tz}$ to $\xi_k^\ast(q)L_k^\ast(t;q)$, which are defined by
    \begin{align*}
        \xi_k^\ast(q) &\coloneqq \sqrt{k}\exp\left(-\frac{z(k-1)}{24}\right)\xi_k(q)
    \end{align*}
    and
    \begin{align*}
        L_k^\ast(t;q) &\coloneqq \frac{K\log z}{tz\sqrt{k}}\exp\left(\frac{z(k-1)}{24}\right).
    \end{align*}
    We see that $L_k^\ast(t;q)$ trivially satisfies hypothesis \ref{item:hypo-1} with $B = 1$ for any $N > 0$. Furthermore, $L_k^\ast(t;q)$ satisfies hypothesis \ref{item:hypo-3} on the minor arc, as for any $\Delta$, if $z = \eeta + iy \to 0$ within the region $\Delta\eeta \leq \abs{y} \leq \pi$, we have that
    \begin{align*}
        \abs{\frac{\log z}{z}} \ll_\Delta \frac{1}{\eeta^2}.
    \end{align*}
    We now verify that $\xi_k^\ast(q)$ satisfies hypotheses \ref{item:hypo-2} and \ref{item:hypo-4}. On the major arc, \Cref{lemma:major-arc-xi} implies that
    \begin{align*}
        \xi_k^\ast(q) = e^{\frac{\pi^2K}{6z}}\left(1 + O\left(e^{-\frac{4\pi^2}{kz}}\right)\right),
    \end{align*}
    and so $\xi_k^\ast$ satisfies \ref{item:hypo-2} with $\beta = 0$, $c^2 = \frac{\pi^2K}{6}$, and $\gamma = \frac{4\pi^2}{k}$. Finally, \Cref{lemma:minor-arc-xi-ngo} immediately implies, on the minor arc, that
    \[
        \abs{\xi_k^\ast(q)} \ll_{k,\Delta} \xi_k^\ast(\abs{q})e^{-\frac{\epsilon}{\eeta}}
    \]
    for some $\epsilon > 0$ when $\Delta$ is chosen sufficiently large. Thus, we may apply \Cref{thm:ngo-circle} which immediately implies the result if we choose $N = 1$.
\end{proof}

Combining \Cref{lemma:akrtn,lemma:cktn} yields \Cref{thm:indiv-asym}.

\section{Proof of \texorpdfstring{\Cref{thm:k-indiv-biases}}{Theorem 1.5}}\label{sec:biases}

\subsection{Estimates for Differences of the Digamma Function}

In this subsection, we provide useful estimates for differences $\psi(a) - \psi(b)$ of the digamma function. Throughout, we only consider the digamma function on $(0, \infty)$. Define for $a > 0$ the function
\[
    \psi_a(x) \coloneqq \psi(x+a) - \psi(x).
\]
Because $\psi$ is increasing, $\psi_a$ is always positive. First note that we may successively apply the recurrence \eqref{eq:psi-recurrence} to see, for any integer $N \geq 1$, that
\begin{align}
    \psi_a(x) = \psi_a(x+N) + \sum_{n=0}^{N-1} \frac{a}{(x+n)(x+a+n)}. \label{eq:n-shift}
\end{align}
In order to use this equation to obtain an infinite series representation of $\psi_a(x)$, we require the following lemma.
\begin{lemma}\label{lemma:diffgamma-dec}
For any $a>0$, $\psi_a$ is decreasing and
\[
\lim_{x\to \infty} \psi_a(x) = 0.
\]
\end{lemma}
\begin{proof}
Using the logarithmic bounds (see \eqref{eq:log-digamma}) for $\psi(x)$, we have
\[
0 < \psi_a(x) < \log\lr{\frac{x+a}{x}} - \frac1{2(x+a)} + \frac{1}{x} < \log\lr{1+\frac ax} + \frac1x,
\]
which implies the desired limit. Let $y > x > 0$. We must prove that $\psi_a(y) < \psi_a(x)$. Write this inequality using \eqref{eq:n-shift} as
\[
\psi_a(x+N) + \sum_{n=0}^{N-1} \frac{a}{(x+n)(x+a+n)} >
\psi_a(y+N) + \sum_{n=0}^{N-1} \frac{a}{(y+n)(y+a+n)}.
\]
We also have
\[
\frac{1}{(x+i)(x+a+i)} > \frac{1}{(y+i)(y+a+i)}
\]
since $y>x>0$.
As $\psi_a$ is arbitrarily close to 0 for large inputs, we choose $N$ large enough to make
\[
\abs{\psi_a(y+N) - \psi_a(x+N)} \leq \abs{\psi_a(x+N)} + \abs{\psi_a(y+N)} <  \frac 1{x(x+a)} - \frac1{y(y+a)},
\]
which proves that $\psi_a$ is decreasing.
\end{proof}

\Cref{lemma:diffgamma-dec} combined with \cref{eq:n-shift} gives us the expansion
\begin{equation}
    \psi_a(x) = a\sum_{m=0}^\infty \frac{1}{(x+m)(x+a+m)}.
    \label{eq:diff-expand}
\end{equation}
We will later require the following estimates for the $\psi$ difference function.
\begin{lemma}
\label{lemma:diffgamma-est}
Let $0 < b < a \leq 1$, then we have the following inequalities:
\[
(a-b)\lr{\frac{1}{ab} + \frac{1}{b+1}} < \psi(a) - \psi(b) < (a-b)\lr{\frac1{ab} + \frac{\pi^2}6}.
\]
Further, if $a, b > \frac12$, then we have
\[
\psi(a) - \psi(b) < (a - b)\lr{\frac1{ab} + \frac{\pi^2}{6} - \frac59}.
\]
\begin{proof}
Using the expansion \eqref{eq:diff-expand} and splitting off the first term, we see that
\begin{align*}
    \psi(a) - \psi(b) &= (a-b)\sum_{m = 0}^\infty \frac{1}{(b+m)(a + m)} \geq (a-b)\left(\frac{1}{ab} + \sum_{m = 1}^\infty \frac{1}{(b+m)(1+m)}\right)
\end{align*}
because $a \leq 1$. Similarly, because $b > 0$ we have that
\begin{align*}
    \psi_(a) - \psi(b) &> (a-b)\left(\frac{1}{ab} + \sum_{m = 1}^\infty \frac{1}{(m+1)(mb + m)}\right) =(a-b)\lr{\frac1{ab} + \frac1{b+1}},
\end{align*}
where the last equality is a simple telescoping sum.
For the other direction we may use that $a,b > 0$ to see that
\begin{align*}
    \psi(a) - \psi(b) &= (a-b)\sum_{m = 0}^\infty \frac{1}{(b+m)(a + m)} \leq (a-b)\left(\frac{1}{ab} + \sum_{m = 1}^\infty \frac{1}{m^2}\right) = (a-b)\left(\frac{1}{ab} + \frac{\pi^2}{6}\right).
\end{align*}
For the case $a, b > \frac12$, if we use the bound
\[
\frac{1}{(a+1)(b+1)} < \frac49,
\]
instead of bounding by 1, we get the result.
\end{proof}
\end{lemma}
When one of the values is $1$, we have stronger bounds
\begin{lemma}\label{lemma:diff-gamma-1}
Let $0 < a < 1$, then we have that
\begin{align*}
    (1-a)\lr{\frac1{a} + \frac{\pi^2}{6} - 1} < \psi(1) - \psi(a) &< (1-a)\lr{\frac1{a} + 1}.
\end{align*}
\end{lemma}

\begin{proof}
As in the proof of \Cref{lemma:diffgamma-est}, write
\begin{align*}
    \psi(1) - \psi(a) &= (1-a)\lr{\frac1a + \sum_{m = 1}^\infty \frac{1}{(m+1)(a+m)}}.
\end{align*}
We then see that because $0 < a < 1$
\begin{align*}
    \frac{\pi^2}{6} - 1 = \sum_{m=1}^\infty \frac{1}{(m+1)^2} \leq \sum_{m = 1}^\infty \frac{1}{(m+1)(a+m)} \leq \sum_{m = 1}^\infty \frac{1}{m(m+1)} = 1. 
\end{align*}
This yields the result.
\end{proof}

\subsection{Proof of \ref{item:r+k} in \Cref{thm:k-indiv-biases}}

Let $k,t \geq 2$ be coprime and let $1 \leq r \leq t - k$. For the proof of \ref{item:r+k}, we may instead show the result when $mk \leq t < (m+1)k$ for every nonnegative integer $m$.
Since $k, t$ are coprime, we have $mk < t$.
Note that when $m = 0$, there is nothing to check, as $t - k \leq 0$, and there is no $1 \leq r \leq t - k$. Thus suppose that $m \geq 1$. As before, we will use \Cref{cor:indiv-diff} and instead show $\psi_{k,t}(r) > \psi_{k,t}(r + k)$. Note that because $r \not\equiv 0 \Mod t$, we have that $\bar{r} \not\equiv 0 \Mod t$, and this implies that $\overline{r + k} = \bar{r} + 1$ because $\bar{r} < t$. 

We first check the lower $\overline{r}$ values, namely when $1 \leq \overline{r} < m$. In this case, we require the following lemma.

\begin{lemma}\label{lemma:r+k bound}
Let $x, a$ be positive real numbers and $k \geq 2$ an integer. Then we have
\[
    k\psi_{ak}(xk) > \psi_a(x).
\]
\end{lemma}

\begin{proof}
We expand both sides by \eqref{eq:diff-expand}, which yields
\begin{align*}
k\psi_{ak}(xk)&=
ak^2\sum_{n=0}^\infty\frac{1}{(xk+n)(xk+ak+n)} \quad \; \text{and} \quad \; \psi_a(x)=
a\sum_{n=0}^\infty\frac{1}{(x+n)(x+a+n)}.
\end{align*}
Comparing term by term, we see that
\begin{align*}
\frac{k^2}{(xk+n)(xk+ak+n)} \geq \frac{1}{(x+n)(x+a+n)} &\iff k^2(x+n)(x+a+n) \geq (xk+n)(xk+ak+n).
\end{align*}
The right hand side holds if and only if
\[
    (xk+nk)(xk+ak+nk) \geq (xk+n)(xk+ak+n) .
\]
However, this inequality holds trivially because everything is positive. In fact, equality only holds when $n=0$. Thus, the inequality is strict for at least one term, implying the result.
\end{proof}
Let $1 \leq \bar{r} \leq m$. Thus $r \equiv k\bar{r} \Mod t$, and $1 \leq k\bar{r} \leq km \leq t$, so $r = k\bar{r}$. Applying \Cref{lemma:r+k bound} with $a = \frac{1}{t}$ and $x = \frac{\bar{r}}{t}$ then proves \ref{item:r+k}, noting that $\overline{r+k} = \bar{r} + 1$ because $m < t$. For $\overline{r} \geq m + 1$, it suffices to show the following lemma, which is a rewritten form of \ref{item:r+k} appearing in \Cref{thm:k-indiv-biases}.
\begin{lemma}
Let $k,t$ be coprime and suppose that $mk \leq t < (m+1)k$ for some $m \geq 2$. If $\bar{r} \geq m + 1$, then we have that
\[
    k\psi_{k/t}\lr{\frac{r}{t}}> \psi_{1/t}\lr{\frac{\overline{r}}{t}}.
\]
\end{lemma}

\begin{proof}
We consider $(k, m, t) = (2, 1, 3)$ separately. Indeed, this can easily be checked numerically. Thus, the $(k, m) = (2, 1)$ case is covered.
    Here we may replace $\overline{r}$ by the lowest value not covered by the first case and $r$ by $t-k$ by \Cref{lemma:diffgamma-dec}. Thus, it suffices to prove
    \[
        k\psi_{k/t}\lr{1 -\frac kt} > \psi_{1/t}\lr{\frac{m+1}{t}}.
    \]
    Using \eqref{eq:diff-expand}, we may write
    \begin{align*}
        k\psi_{k/t}\lr{1-\frac kt} &= k^2 \sum_{j=1}^\infty\frac{1}{j(jt-k)}, && \psi_{1/t}\lr{\frac{m+1}{t}} =  \sum_{j=1}^\infty \frac{t}{(jt-t + m+1)(jt - t + m+2)}.
    \end{align*}
    If we look at each term, it suffices to show that
    \[
    \frac{k^2}{j(jt-k)} > \frac{t}{(jt-t + m+1)(jt - t + m+2)}
    \]
    or, equivalently,
    \[
        k^2 > \frac{jt}{jt - t + m+2} \cdot \frac{jt - k}{jt-t + m+1}.
    \]
    Because $k \geq 2$, we have
    \[
    jt(jt-k) < (jt-1)(jt-k+1),
    \]
    which upon expanding and cancelling terms we are left with $k > 1$. Thus we need to prove that,
    \[
    k^2 > \frac{jt-1}{jt - t + m+2} \cdot \frac{jt - k+1}{jt-t + m+1}.
    \]
    We now see that if $m \geq 2$, then $km \geq m+2$, which implies $t \geq km+1 \geq m+3$. If $m=1$, then $k \geq 3$ would imply $t \geq km+1 \geq m+3$. Thus we have that
    \[
        \frac{jt-1}{jt-t + m+2}
    \]
    is maximized when $j = 1$. Similarly, we have that the second term
    \[
    \frac{jt - k+1}{jt - t + m+1},
    \]
    is maximized when $j=1$ if the numerator is larger than the denominator. That is we must have $k + m \leq t$. But notice that we have $t \geq km+1 = (k-1)(m-1) + k + m \geq k + m$ because $m\geq 1$. Thus, we only need to check the case when $j=1$, which may be written as
    \[
     k^2 > \frac{t(t - k)}{(m+1)(m+2)}.
    \]
    We may then cross-multiply to write this inequality as
    \[
    (m+1)(m+2) > \frac tk \lr{\frac tk - 1},
    \]
    which holds because $t < (m+1)k$.
\end{proof}

As before, the above shows \ref{item:r+k} due to \Cref{cor:indiv-diff} and the fact that $\overline{r + k} = \bar{r} + 1$.

\subsection{Proof of \ref{item:fixing} and \ref{item:t2} in \Cref{thm:k-indiv-biases}}

We now prove \ref{item:fixing}, and in the process prove \ref{item:t2}. Let $1 \leq r \leq y \leq t$, $r < s \leq t$, and $k,t \geq 2$ be coprime. In light of \Cref{cor:indiv-diff}, it suffices to show that $\psi_{k,t}(r) > \psi_{k,t}(s)$. In other words, we must show that
\begin{align*}
    \psi\left(\frac{s}{t}\right) - \psi\left(\frac{r}{t}\right) > \frac{1}{k}\left(\psi\left(\frac{\bar{s}}{t}\right) - \psi\left(\frac{\bar{r}}{t}\right)\right),
\end{align*}
where $1 \leq \bar{r} \leq t$ is the representative of $rk^{-1}$ mod $t$, and similarly for $\bar{s}$. The left hand side is minimized when $s = y+1, r = y$ whereas the right hand side is maximized when $\bar{s} = 1, \bar{r} = \frac{1}{t}$. Thus we may instead show that
\begin{align*}
    \psi\left(\frac{y+1}{t}\right) - \psi\left(\frac{y}{t}\right) > \frac{1}{k}\left(\psi\left(1\right) - \psi\left(\frac{1}{t}\right)\right).
\end{align*}
\Cref{lemma:diff-gamma-1} implies that
\[
    \psi\left(1\right) - \psi\left(\frac{1}{t}\right) < \frac{t^2 - 1}{t}.
\]
Likewise, \Cref{lemma:diffgamma-est} implies that
\[
    \psi\left(\frac{y+1}{t}\right) - \psi\left(\frac{y}{t}\right) > \frac{1}{t}\left(\frac{t^2}{(y+1)y} + \frac t{y + t}\right) > \frac{1}{t}
    \lr{\frac{2t^2 + y(y+1)}{2y(y+1)}} >
    \frac 1t\lr{\frac{t^2-1}{y(y+1)}}.
\]
Therefore, it suffices to take
\begin{align*}
    k \geq y(y+1),
\end{align*}
proving \ref{item:fixing}.

Now we prove \ref{item:t2}. When $y = t-1$, we may use \Cref{lemma:diff-gamma-1} instead of \Cref{lemma:diffgamma-est} to see that
\begin{align*}
    \psi\left(1\right) - \psi\left(\frac{t-1}{t}\right) > \frac{1}{t}\left(\frac{t}{(t-1)}  + \frac{\pi^2}{6} - 1\right) > \frac{\pi^2}{6t}.
\end{align*}
Thus it suffices to take $k \geq \frac{6(t^2 - 1)}{\pi^2}$, proving \ref{item:t2}.

\subsection{Proof of \ref{item:high-unnatural} in \Cref{thm:k-indiv-biases}}

Let $k = mt - 1$, then we wish to show that $\psi_{k,t}(t) > \psi_{k,t}(t-1)$ for sufficiently small $m$. Note that for any $1 \leq r \leq t-1$, since $k\equiv -1 \Mod t$, we have $\bar{r} = t - r$, and we have $\bar{t} = t$. Thus, we may rewrite this inequality as
\begin{align*}
    \frac{1}{k}\left(\psi\left(1\right) - \psi\left(\frac{1}{t}\right)\right) > \psi(1)-\psi\left(\frac{t-1}{t}\right).
\end{align*}
Applying \Cref{lemma:diff-gamma-1} and \cref{eq:diff-expand}, it suffices to show that
\begin{align*}
    \frac{t-1}{tk}\left(t + \frac{\pi^2}{6} - 1\right) > \frac{1}{t}\sum_{n = 0}^\infty \frac{1}{(\frac{t-1}{t} + n)(n + 1)}.
\end{align*}
We then see that, since $\frac{\pi^2}{6} = \sum_{n = 1}^\infty \frac{1}{n^2}$,
\begin{align*}
    -\frac{\pi^2}{6} + \sum_{n = 0}^\infty \frac{1}{(\frac{t-1}{t} + n)(n + 1)} &= \frac{1}{t}\sum_{n=0}^\infty \frac{1}{(n+1)^2\left(\frac{t-1}{t} + n\right)} \leq \frac{1}{t-1} + \frac{1}{t}\sum_{n=1}^\infty \frac{1}{(n+1)^2 n} \\
    &= \frac{1}{t-1} + \frac{2}{t} - \frac{\pi^2}{6t} \leq \frac{5}{2t}.
\end{align*}
Thus, it suffices for $k$ to satisfy
\begin{align*}
    \frac{t-1}{tk}\left(t + \frac{\pi^2}{6} - 1\right) > \frac{1}{t}\left(\frac{\pi^2}{6} + \frac{5}{2t}\right).
\end{align*}
Rearranging, we have that $\psi_{k,t}(t) > \psi_{k,t}(t-1)$ whenever
\begin{align*}
    k < (t-1)\left(t + \frac{\pi^2}{6} - 1\right)\left(\frac{\pi^2}{6} + \frac{5}{2t}\right)^{-1}.
\end{align*}
Estimating then yields the result.

\subsection{The Proof of \ref{item:supralinear} in \Cref{thm:k-indiv-biases}}
\label{subs:linear}

Now we prove \ref{item:supralinear}. To do so, we show that for $k \in (t/2, t)$, the least common residue is $k \Mod t$; in other words, for $1 \leq r \leq t$ and $r\neq k$, then $r\succ_{k, t} k$.

Assume that $k \in \lr{\frac{m-1}{m}t, \frac{m}{m+1}t}$ for $2 \leq m < t$. Doing so will not exclude any allowed $\frac t2 < k < t$ because $k = \frac{m-1}m t$ for some $m$ would violate coprimality with $t$. We divide into 2 cases: $\overline{r} \leq m$ and $\overline{r} \geq m+1$. Note that since $r \neq k$ and $k < t$, $\bar{r} \neq 1$.

Suppose that $2 \leq \rbar \leq m$, then we have $1 - \frac1\rbar \leq 1 - \frac1m < \frac kt < 1$, which implies
\[
0 < \rbar k - (\rbar - 1)t < t.
\]
Thus, $r = \rbar k - (\rbar - 1)t$. Moreover, we have $(\bar{r}-1)k < (\bar{r}-1)t$, and so
\[
r=\rbar k - (\rbar - 1)t < k.
\]
Thus, we obtain that
\[
\psi_{k, t}(r) = -\psi\lr{\frac rt} + \frac1k \psi\lr{\frac\rbar t} > -\psi\lr{\frac kt} + \frac1k\psi\lr{\frac1t} = \psi_{k,t}(k)
\]
because $\psi$ is increasing.

Now assume $\rbar > m$.
We need to prove that
\[
\psi_{k, t}(r) > \psi_{k, t}(k),
\]
which is equivalent to
\[
\psi\lr{\frac \rbar t} - \psi\lr{\frac 1t} > k\lr{\psi\lr{\frac rt} - \psi\lr{\frac kt}}.
\]
Assume for the moment that $m > 2$, and we will later handle the case of $m = 2$ separately. Because $\rbar \geq m+1, r \leq t$ and $\psi$ is increasing, it suffices to show the  inequality
\[
\psi\lr{\frac{m+1}t} - \psi\lr{\frac1t} > k\lr{\psi(1) - \psi\lr{\frac kt}}.
\]
Using \Cref{lemma:diffgamma-est}, the left hand side is bounded by
\begin{align*}
    \psi\lr{\frac {m+1}t} - \psi\lr{\frac1t} &> \frac mt\lr{\frac{t^2}{m+1} + \frac{t}{t + 1}} = \frac{mt}{m+1} + \frac{m}{t+1} > \frac{mt}{m+1}.
\end{align*}
Using \Cref{lemma:diff-gamma-1}, the right hand side is bounded by
\[
k\lr{\psi(1) - \psi\lr{\frac kt}} < k\frac{t-k}t\lr{\frac tk + 1} = \frac{t^2 - k^2}{t} < t\lr{1 - \frac{(m-1)^2}{m^2}}.
\]
Thus, we need to prove
\[
\frac{mt}{m+1} > t\lr{1 - \frac{(m-1)^2}{m^2}},
\]
which is equivalent to
\[
\frac{(m-1)^2}{m^2} > \frac{1}{m+1}.
\]
This holds when $m \geq 3$, as
\[
(m-1)(m^2 - 1) \geq 2(m^2 - 1) > m^2.
\]

Now we consider $m=2$. If $\rbar = 3$, we have
\[
\psi\lr{\frac3t} - \psi\lr{\frac1t} > \frac2t\cdot\frac{t^2}3 = \frac{2t}{3}.
\]
The condition $k \in (\frac t2, \frac{2t}3)$ implies that $0 < 3k -t < t$, which gives $r = 3k-t$. Thus we have $\frac rt>\frac kt > \frac 12$, which implies
\[
\psi\lr{\frac rt} - \psi\lr{\frac kt} < \frac{2k-t}{t}\lr{\frac{t^2}{(3k-t)k} + \frac{\pi^2}6 - \frac 59},
\]
using the special case of \Cref{lemma:diffgamma-est}.
The bounds on $k$ imply that
\[
\frac{(2k-t)t}{(3k-t)k} < \frac{1}{2}.
\]
We can see this by clearing denominators and factoring, which gives
\[
0 < (2t-3k)(t-k),
\]
where we know both terms are positive. Thus, we find that
\[
k\lr{\psi\lr{\frac rt} - \psi\lr{\frac kt}} < k\lr{\frac12 + \lr{\frac{\pi^2}{6} - \frac 59} \frac{2k-t}t} < \frac{2t}3.
\]
If $\rbar \geq 4$, then we have
\[
\psi\lr{\frac \rbar t} - \psi\lr{\frac1t} > \frac{3t}4.
\]
For the other side, we have
\[
k\lr{\psi\lr{\frac rt} - \psi\lr{\frac kt}} < k\lr{\psi(1) - \psi\lr{\frac kt}} < \frac{t^2 - k^2}t < \frac{3t}4,
\]
which finishes the proof that $k$ is the least common residue class modulo $t$. Thus, for each $k \in (\frac t2, t)$, the ordering will be distinct. When $t>2$, there are exactly $\frac{\varphi(t)}2$ numbers in $(\frac t2, t)$ that are coprime to $t$ because $\gcd(i, t) = \gcd(t-i,t)$. Thus, there are at least $\frac{\varphi(t)}2$ distinct orderings.
\begin{remark}
Concerning \Cref{conj:superlinear}, the following heuristic points to a strategy which might be used to prove the conjecture.
If $k' = k + mt$ for integer $m$, then $\psi_{k, t}$ and $\psi_{k', t}$ differ only by the weighing factors $\frac 1k$ and $\frac 1{k'}$. Thus, if $r < s$ and $r\prec_{k, t} s$, then there exists a ``switching point'' $S_k$ such that $r\prec_{k', t} s$ when $k' < S_k$ and $r\succ_{k', t} s$ when $k' \geq S_k$.
If for a set of pairs $(r, s)$ these switching points are sufficiently spaced out to occur at different values of $k' = k + tm$, then these $k'$ induce distinct orderings.
The difficulty of this approach lies in finding suitable pairs $(r, s)$ that have convenient switching points. However, data suggests that it is possible to find a set of pairs $(r, s)$ which has cardinality linear in $m$, after which a simple calculation would yield that $\frac{\mathscr{O}(t)}{\varphi(t)}$ is at least on the order of $\log(t)$.
\end{remark}

\printbibliography

\end{document}